\documentclass{amsart}

\usepackage{graphicx}
\usepackage{booktabs}
\usepackage{multirow}

\newtheorem{theorem}{Theorem}[section]
\newtheorem{corollary}{Corollary}[theorem]
\newtheorem{lemma}[theorem]{Lemma}

\theoremstyle{definition}

\theoremstyle{remark}

\numberwithin{equation}{section}

\begin{document}
\

\title{Subtree Size in Various Planar Trees}

\author{Anthony Van Duzer}
\address{Department of Mathematics, University of Florida, Gainesville, Florida 32601}
\curraddr{Department of Mathematics,
University of Florida, Gainesville, Florida 32611}
\email{avanduzer@ufl.edu}

%    General info
%\subjclass[2000]{Primary 54C40, 14E20; Secondary 46E25, 20C20}

%\date{January 1, 2001 and, in revised form, June 22, 2001.}


\begin{abstract}
In this paper we find the generating function for the number of vertices which have $k$ elements in their subtree and use this generating function to calculate the probability that a vertex has a size $k$ subtree.  We also show how this same technique can be applied to calculate the probabilities for other trees and specifically apply it to $4$ different types of trees.
\end{abstract}

\maketitle

\section{Introduction}
In the past few decades there have been several papers looking at various statistics related to vetices in both labelled and unlabelled trees.  The traditional method of doing this was to get a generating function for the number of trees where the root has the property you are intereseted in and once you have that you can use it to find the number of vertices have the property you are intersted in.  We will apply this technique to 2 relatively novel vertex statistics.  Specifically we will look at the probability a vertex has $k$ leaves in its subtree and the probability a vertex has $k$ vertices in its subtrees.
\vspace{.1in}
\newline
We will do this by finding the ordinary generating function, in this paper generating function will always refer to ordinary genterating functions of the form $A(x)=\displaystyle\sum_{n=0}^\infty a(n) x^n,$ for the total number of vertices that have that property and use that to calculate the asymptotic behavior of $a(n)$ and compare that to the total number of vertices in all trees of size $n$ as $n$ goes to $\infty$.  
\vspace{.1in}
\newline
To find the generating function for the total number of vertices with a given property we will first need to find the generating function for the number of trees where the root has that specific property.  Once we do that we can then construct the generating function we are interested in. For all of the trees in this paper the generating function for one of the two classes we are interested in will be relatively straight forward to calculate.  For trees counted by the number of vertices the generating function for the number of trees where the root has $n$ vertices in its subtree is just $A_n x^n$ where $A_n$ is the total number of trees of size $n$.  A similar relation holds for trees which are counted by the number of leaves and the generating function for the number of trees where the root has $n$ leaves.That leaves $2$ casess, the case where the trees are counted by the number of leaves and we are interested in the generating function for the number of trees with $n$ vertices and the case where the tree are counted by the number of vertices and we wish to know the generating function for the number of trees with $n$ leaves. To calculate these generating function we will construct a bivariate generating function where $x^n$ is indexed by vertices and $y^n$ is indexed by leaves and we will look at the coefficent of $x^k$ if we are interested in vertices and $y^k$ if we are interested in leaves.  This will just be a generating function in $y$ or $x$ respectively.  For the case of trees counted by the number of leaves we will also have to calculate the total number of vertices.

\section{Motzkin Trees}
The first example we will look at are Motzkin trees.  A Motzkin tree is a rooted planar tree where each non-leaf vertex has either one or two children.  These are also occasionally referred to as unary-binary trees or 1-2 trees.
\begin{theorem}
Let $M(x)=\displaystyle\sum_{n=0}^\infty m(n) x^n$ be the generating function for the number of Motzkin Trees.  We have that \begin{equation}M(x)=\dfrac{1-x-\sqrt{1-2x-3x^2}}{2x}.\end{equation}
\end{theorem}
\begin{proof}
This comes from the relationship $M(x)=x+xM(x)+xM(x)^2$, which is based on the fact the root can be a leaf, the parent of a single child, or the parent of two children and each child of the root would be another Motzkin tree.  Given that relationship we can use the quadratic formula to arrive at the desired generating function.
\end{proof}
\subsection{$k$ verties in the subtree}
We will first handle the case for $k$ vertices in the subtree.
\begin{theorem}
The generating function for the number of vertices that have $k$ vertices in their subtree in all trees with $n$ vertices is given by \begin{equation}V_k(x)=\dfrac{R_k(x)}{\sqrt{1-2x-3x^2}}\end{equation} where $V_k(x)$ is the generating function for all vertices with $k$ vertices in their subtree in all Motzkin tree with $n$-vertices and $R_k(x)$ is the number of trees where the root has $k$ vertices in its subtree.
\end{theorem}
\begin{proof}
This follows from the recurrence relationship \begin{equation}V_k(x)=R_k(x)+xV_k(x)+2xV_k(x)M(x).\end{equation}  To see why this recurrence relationship holds consider the root.  The root contributres $R_k(x)$ to the generating function, by definition.  Now that we have handled the root consider what happens when we remove the root.  The root will either have 1 or 2 children.  If it has 1 child than that child will be a Motzkin tree and the generating function for the number of vertices with $k$ vertices in their subtree is simply $V_k(x)$ and we multiply by $x$ to account for the root.  Otherwise it will break into 2 Motzkin trees.  This contributes the $2xV_k(x)M(x)$ term.  To see why fix a Motzkin tree on the left subtree and color one of the vertices with $k$ vertices in its subtree red.  We want to see what that contributes to the generating function.  If the left subtree is of size $m$ and the right subtree is of size $r$ the vertex will contribute $x^{m+1}x^{r}$ to the generating function the $m+1$ because of the root and the $x^r$ to account for the size of the right subtee for every occurence.  Now we have to see how often this $1$ vertex will appear.  The right subtree is simply another Motzkin tree so this one vertex will contribute $x^{m+1}M(x)$ to the generating function.  We now have to sum over all possible configurations and vertices we could chose on the left side which gives us the term $xV_k(x)M(x)$ and finally the 2 comes from the fact we have to do the same thing to account for the right subtree.
\vspace{.1in}
\newline
Plugging in the known $M(x)$ and solving for $V_k(x)$ gives us the desired fucntion.
\end{proof}
Now that we have that relationship we merely need to find $R_k(x)$ but since Motzkin trees are counted by the number of vertices then the number of trees where the root has $k$ vertices in its subtree is simply $M_k$ and they all have size $k$ so $R_k(x)=M_kx^k$.  To go from these generating functions to probabilities we will need to use Bender's lemma.
\begin {lemma} (Bender's Lemma) \cite{D}
\newline
Take generating functions $A(x)=\sum a_nx^n$ and $B(x)=\sum b_n x^n$ with radius of convergence $\alpha >\beta \geq 0$ where $\alpha$ goes with $A(x)$ and $\beta$ goes with $B(x)$.  If $\frac{b_{n-1}}{b_n}$ approaches a limit b as n approaches infinity and $A(b) \neq 0$ then $c_n \sim A(b)b_n$ where $\sum c_nx^n$=A(x)B(x).
\end {lemma}
When we apply Bender's lemma we have that $A(x)=R_k(x)$ and \begin{equation}B(x)=\dfrac{1}{\sqrt{1-2x-3x^2}}.\end{equation}  With some algebra we see that $B(x)$ has a radius of convergence of $\frac{1}{3}.$  We have that  $R_k(x)$ is a non-zero polynomial in all these cases so has an infinite radius of convergence and hence we can apply Bender's lemma.
\begin{center}
 \begin{tabular}{|c | c| c|} 
 \hline
k & $R_k(x)$ & Probability the subtree\\&& has $k $ vertices $\approx$\\ [.5ex] 
\hline
 1 & $x$ & 0.33333333\\ 
\hline
 2 & $2x^2$ & 0.22222222 \\
\hline
3 & $4x^3$ &  0.14814815\\
\hline
4 & $9x^4$ &  0.11111111\\
\hline
5 & $21x^5$ & 0.086419753\\
\hline
6 & $51x^6$ &  0.069958848\\
\hline
\end{tabular}
\end{center}
\subsection{$k$ leaves in the subtree}
Next we will handle the question of $k$ leaves in the subtree.  This case is slightly more complicated to handle.
\begin{lemma}
Letting $L_k(x)=$ the generating function for the number of vertices in all Motzkin trees with $n$ vertices which have $k$ leaves in their subtree we  have \begin{equation}L_k(x)=\dfrac{R'_k(x)}{\sqrt{1-2x-3x^2}}\end{equation} where $R'_k(x)$ is the generating function for the number of roots which have $k$ leaves in their subtree.
\end{lemma}
The proof of this is identical to the proof we had for $k$ vertices, and we simply need to find the generating function for the number of roots which have $k$ leaves.  To find $R'_k(x)$ we will need to set up a bivariate generating function,
\begin{lemma}
Let $M(x,y)$ be the bivariate generating functions for Motzkin trees where the number of vertices is indexed by $x$ and the number of leaves is indexed by $y$.   We have that \begin{equation}M(x,y)=\dfrac{1-x-\sqrt{1-2x+x^2-4x^2y}}{2x}.\end{equation}
\end{lemma}
\begin{proof}
We have the relationship \begin{equation}M(x,y)=xy+xM(x,y)+xM^2(x,y)\end{equation} and we simply solve for $M(x,y)$ using the quadratic formula.
\end{proof}
Now that we have that bivariate generating function to find the number of trees where the root has $k$ leaves we simply need to extract the coefficient of $y^k$ from the generating function and using that and Bender's Lemma we can construct the following table.
\begin{center}
 \begin{tabular}{|c | c| c|} 
 \hline
k & $R'_k(x)$ & Probability the subtree\\&& has $k $ leaves $\approx$\\ [.5ex] 
\hline
 1 & $\frac{x}{1-x}$ & 0.5\\ [1ex]
\hline
 2 & $\dfrac{x^3}{(1-x)^3}$ & 0.125 \\ [1ex]
\hline
3 & $\dfrac{2x^5}{(1-x)^5}$ &  0.0625\\ [1ex]
\hline
4 & $\dfrac{5x^7}{(1-x)^7}$ &  0.0391\\ [1ex]
\hline
5 & $\dfrac{14x^9}{(1-x)^9}$ & 0.02734\\ [1ex]
\hline
6 & $\dfrac{42x^{11}}{(1-x)^{11}}$ &  0.02051\\ [1ex]
\hline 
\end{tabular}
\end{center}
We see that all of the $R'_k(x)$  have a radius of convergence of 1 so Bender's lemma applies in this case.
\section{Ordered Trees} 
The second example we will look at is ordered trees with no degree restrction counted by the number of vertices.  The first thing we need to do is find the total number of vertices in all ordered trees of size $n$.  This is an incredibly well studied class of trees and we know the class is enumerated by the Catalan numbers.
\begin{theorem}
The generating function for the number of trees \begin{equation}T(x)=\dfrac{1-\sqrt{1-4x}}{2}\end{equation}
\end{theorem}
\begin{proof}
We have the relationship \begin{equation}T(x)=x+xT(x)+xT^2(x)+...\end{equation} which leads to \begin{equation}T(x)=\dfrac{x}{1-T(x)}\end{equation}  We then solve for $T(x)$ and get \begin{equation}T(x)-T^2(x)-x=0\end{equation} and solve this using the quadratic formula.  We know that $t_1=1$ since there is a single rooted plane tree on 1 vertex.
\end{proof}
This lacks the $x$ in the denominator of the standard generating function, because we are enumerating based on the number of vertices and there are no trees on $0$ vertices and a single tree on $1$ vertex and a single tree on $2$ vertices.  
\begin{lemma}
If we let $V(n)=$ the total number of vertices in all trees of size $n$ then we have $V(n)\sim \binom{2n+2}{n+1}$ or $V(n) \sim \dfrac{4^{n+1}}{\sqrt{(n+1)\pi}}$
\end{lemma}
Now that we have the total number of vertices we can start looking at the statitistics we are interested in.
\subsection{$k$ verties in the subtree}
Similarly to the case for Motzkin trees we again want to find a generating function for this based on the generating function for the number of roots that have this statistic.
\begin{theorem}
Let $L_k(x)=$ the generating function for the number of vertices in all rooted planar trees on $n$ vertices which have $k$ vertices in their subtree.  Then we have that \begin{equation}{L_k(x)=R_k(x)\left(\frac{1}{2}\left(1+\dfrac{1}{\sqrt{1-4x}}\right)\right)}.\end{equation}
\end{theorem}
\begin{proof}
We have the relationship that \begin{equation}L_k(x)=R_k(x)+L_k(x)+2L_k(x)T(x)+3L_k(x)T(x)^2+4L_k(x)T(x)^3+...\end{equation} and by simplyifing that we get that  \begin{equation}L_k(x)=R_k(x)+\dfrac{L_k(x)}{(1-T(x))^2}.\end{equation}  Plugging in the known value for $T(x)$ and solving gives us the desired generating function.
\end{proof}
We know that half of all vertices in all ordered trees are leaves so we can use that fact combined with Bender's Lemma to construct the following table of probabilities.
\begin{center}
 \begin{tabular}{|c | c| c|} 
 \hline
k & $R_k(x)$ & Probability the subtree\\ &&has $k$ leaves $\approx$\\ [.5ex] 
 \hline
 1 & $x$ & .5\\ 
 \hline
 2 & $x^2$ & .125 \\
 \hline
3 & $2x^3$ & 0.0625 \\
\hline
4 & $5x^4$ & 0.03906\\
\hline
5 & $14x^5$ & 0.02734\\
\hline
6 & $42x^6$ & 0.02051\\
\hline
7 & $132x^7$ & 0.0161133\\
\hline
\end{tabular}
\end{center}
\subsection{$k$ leaves in the subtree}
Much like in the case of Motzkin trees we will have to construct a bivariate generating function here to get the number of trees where the root has $k$ leaves.
\begin{lemma}
Let $T(x,y)$ be the generating function for the number of ordered trees where $x$ counts the number of vertices and $y$ counts the number of leaves.  Then we have \begin{equation}T(x,y)=\dfrac{(xy-x+1)-\sqrt{x^2y^2-2xy^2+x^2-2xy-2x+1}}{2}.\end{equation}
\end{lemma}
\begin{proof}
We have the relationship \begin{equation}T(x,y)=xy+xT(x,y)+xT(x,y)^2+xT(x,y)^3+...\end{equation} This is because an ordered tree either just consist of a single vertex which is both the root and a leaf or we can cut off the root and it decomposes into a forest of ordered trees with any number of ordered trees.  The terms $xT(x,y)+xT(x,y)^2+xT(x,y)^3+...$ form a gemoetric series with common ratio $T(x,y)$ so by simplifying we get that \begin{equation}T(x,y)^2-(xy-x+1)T(x,y)+xy=0.\end{equation}  We can then solve for $T(x,y)$ using the quadratic formula and we obtain the desired generating function.
\end{proof}
Now that we have this generating function we can use it and Bender's lemma to construct the following table of probabilities
\begin{center}
 \begin{tabular}{|c | c| c|} 
 \hline
k & $L_k(x)$ & Probability the subtree\\&& has $k $ leaves $\approx$\\ [.5ex] 
\hline
 1 & $\dfrac{x}{1-x}$ & 0.666666667\\ 
\hline
 2 & $\dfrac{x^3}{(1-x)^3}$ & 0.07407407 \\
\hline
3 & $\dfrac{x^4+x^5}{(1-x)^5}$ &  0.04115226\\
\hline
4 & $\dfrac{x^5+3x^6+x^7}{(1-x)^7}$ &  0.0265203475\\
\hline
\end{tabular}
\end{center}

\section{Full Binary Trees}
Here we will look at full binary trees counted by the number of leaves.  This example is slightly different than the two preceeding examples because rather than counting by the total number of vertices we are going to count by the number of leaves.  Again the first step we need to do is calculate the total number of vertices.  This another example of a class of trees enumerated by the Catalan numbers.  There is one significant difference between these and ordered trees however.  For ordered trees we were counting trees by the number of vertices here we are counting by the number of leaves.  Thankfully for a full binary tree on $n$ leaves there are $n-1$ internal vertices so the total number of vertices is simply $2n-1$
\newline
So as with our previous examples we need to calculate the generating function for the number of such trees but this is again another famous example of the Catalan numbers so we know that the generating function is \begin{equation}B(x)=\dfrac{1-\sqrt{1-4x}}{2}.\end{equation}
\subsection{$k$ verties in the subtree}
We again will calculate the desired generating function and we get that.
\begin{theorem}
Let $V_k(x)=$the generating function for the number of vertices in all full Binary trees with $k$ vertices in its subtree.  Then we have that \begin{equation}V_k(x)=\dfrac{R_k(x)}{\sqrt{1-4x}}\end{equation} where $R_k(x)$ is the generating function for the number of roots with $k$ vertices in its subtree.                                                                                                                                                                                                                                                                                                                                                                                                                                                                                                                                                                                            
\end{theorem}
\begin{proof}
We have the recurrence relationship \begin{equation}V_k(x)=R_k(x)+2V_k(x)B(x).\end{equation}  This can be found using the same argument we used for Motzkin trees.
\end{proof}
\begin{corollary}
The generating function for the number of leaves is \begin{equation}L(x)=\dfrac{x}{\sqrt{1-4x}}\end{equation}  A leaf is a vertex with only a single vertex in its subtree. So the generating function for the number of trees where the root is a leaf is simply $x$.
\end{corollary}
Asymptotically half of all vertices in a full binary tree are leaves.  Using that information and Bender's lemma we can construct the following table of probabilities.
\vspace{.1in}
\newline
Since this tree is counted by the number of leaves rather than the number of vertices we would normally need to use a bivariate generating function here to get the number of trees where the root has $k$ vertices, but that is actually not necessary in this case.  A full binary tree on $n$ leaves will have $2n-1$ vertices.  So if we want to find the number of trees with $k$ vertices than $k$ must be odd and of the form $2n-1$ and in that case it will simply have $n$ leaves so the generating function for roots with $k$ vertices is simply $B_{\left(\frac{k+1}{2}\right)}x^{\frac{k+1}{2}}.$ 
\begin{center}
 \begin{tabular}{|c | c| c|} 
 \hline
k & $R_k(x)$ & Probability the subtree\\ &&has $k$ vertices $\approx$\\ [.5ex] 
 \hline
 1 & $x$ & .5\\ 
 \hline
 2 & 0 & 0 \\
 \hline
3 & $x^2$ & 0.125 \\
\hline
4 & 0 & 0\\
\hline
5 & $2x^3$ & 0.0625\\
\hline
6 & $0$ & 0\\
\hline
7 & $5x^4$ & 0.0161133\\
\hline
\end{tabular}
\end{center}
We get 0 for all even values of $k$ because a full binary tree cannot have an even number of vertices.  We cannot technically apply Bender's lemma in those cases, but such machinery is unnecessary as clearly the probabilities are 0 in those cases since there can be no vertices with an even number of vertices in their subtree.
\subsection{$k$ leaves in the subtree}
Since these trees are counted by the number of leaves the generating function for the number of trees where the root has $k$ vertices is simply $B_k x^k$.  From that we can construct the following table of probabilities.
\begin{center}
 \begin{tabular}{|c | c| c|} 
 \hline
k & $R_k(x)$ & Probability the subtree\\ &&has $k$ leaves $\approx$\\ [.5ex] 
 \hline
 1 & $x$ & .5\\ 
 \hline
 2 & $x^2$ & .125 \\
 \hline
3 & $2x^3$ & 0.0625 \\
\hline
4 & $5x^4$ & 0.03906\\
\hline
5 & $14x^5$ & 0.02734\\
\hline
6 & $42x^6$ & 0.02051\\
\hline
7 & $132x^7$ & 0.0161133\\
\hline
\end{tabular}
\end{center}
This example did not get to show off the technique to its full capability, since, while ostensibly the size of the trees was based on the number of leaves, in this specific tree the number of leaves and the number of vertices are directly related.  The number of vertices with $k$ leaves in their subtree would simply be the number of vertices with $2k-1$ vertices in its subtree.  In the next section we will see an example of a tree counted by leaves where there isn't a direct relationship between the number of vertices and the number of leaves.

\section{Schroeder Trees} A Schroeder tree of size $n$ is a rooted planar tree on $n$ leaves in which each non-leaf vertex has at least 2 children.  This problem has one substantial difference from the problem discussed in the precedding section.  In the case of full binary trees a tree with $n$ leaves had $2n-1$ vertices.  So it was quite easy to find the total number of vertices and the number of trees with $m$ vertices.  In this case a tree with $n$ leaves can have a range of values for the number of vertices.  So in this case we will need to use a bivariate generating function to find the number of vertices and to find the generating function for the number of trees on $k$ leaves that have $m$ vertices.
\vspace{.1in}
\newline
As with all the previous examples the first step is to find the generating function for the total number of trees.
\begin{theorem}
The generating function, $S(x)=\displaystyle\sum s_n x^n$, for the number of all Schroeder trees with n leaves is given by \begin{equation}S(x)=\dfrac{2x}{1+x+\sqrt{1-6x+x^2}}\end{equation}
\end{theorem}
\begin{proof}
This comes from the relationship \begin{equation}S(x)=x+(S(x))^2+(S(x))^3+(S(x))^4+...\end{equation} which gives us \begin{equation}S(x)=x+\dfrac{(S(x))^2}{1-S(x)}.\end{equation}  This relationship is based on the fact the root can be a leaf or the parent of at least two children.  If it is a leaf it contributes $x$ and if it is the parent of n children each child would be another Schroeder tree, we do not multiply by $x$  since we are counting by leaves and the root would not be a leaf.  Solving this equation we get the desired generating function.
\end{proof}
\begin {corollary}
Let $l(n)$ be the number of leaves in all Schroeder trees of size n.  Then we have \begin{equation}l(n) \sim \dfrac{(1+\sqrt{2})(n-1)^{(\frac{-1}{2})}}{2^{\frac{7}{4}}\sqrt{\pi}}\left(3+\sqrt{8}\right)^{n-1}.\end{equation}
\end{corollary}
\begin{proof}
We can extract the coefficent from the generating function and we get that \begin{equation}S(n) \sim \dfrac{(1+\sqrt{2})(n-1)^{(\frac{-3}{2})}}{2^{\frac{7}{4}}\sqrt{\pi}}\left(3+\sqrt{8}\right)^{n-1},\end{equation} and there are precisely $n$ leaves on each Schroeder tree, so we simply multiply that number by $n$.
\end{proof}
We will now need to find the total number of vertices and unlike the previous several examples a tree of size $k$ can have many different number of vertices so we will need to use a more advaned method to calculate the number of vertices.  To do this we will calculate a bivariate generating function.
\begin{theorem}
The generating function for the number of vertices in all Schroeder trees with  $n$ leaves is given by \begin{equation}V(x)=\dfrac{3-x+\sqrt{1-6x+x^2}}{4\sqrt{1-6x+x^2}}\left(\dfrac{2x}{1+x+\sqrt{1-6x+x^2}}\right).\end{equation}
\end {theorem}
\begin{proof}
We set up a bivariate generating function where $x $ counts the leaves and $y$ counts the vertices.  We call this $V(x,y)$ from there, if we differentiate with respect to $y$ and then set $y=1$ we get the univariate generating function for the number of all vertices in trees with $n$ leaves.  We have the relationship \begin{equation}V(x,y)=xy+y(V(x,y))^2+y(V(x,y))^3+y(V(x,y))^4+...\end{equation} which gives us \begin{equation}V(x,y)=xy+\dfrac{y(V(x,y))^2}{1-V(x,y)}.\end{equation}  Solving for $V(x,y)$ we get \begin{equation}V(x,y)=\dfrac{1+xy-\sqrt{(xy)^2+2xy+1-4xy(y+1)}}{2y+2}.\end{equation}  We then differentiate with respect to $y$ and substitute in $y=1$ to get the desired generating function.
\vspace{.1in}
\newline
We will see an arguably simpler way we coulod have calculated this later but the bivariate generating function we calculated here will be needed for later work.
\end{proof}
\begin {corollary}
Let $V(n)$ be the number of all vertices in all Schroeder trees with $n$ leaves.  We have $V(n) \sim \dfrac{(3+\sqrt{8})^n}{2^{\frac{9}{4}}\sqrt{\pi n}}$.
\end{corollary}
\begin{proof}
We can extract the coefficent from the previous generating function giving us these asymptotics.
\end{proof}
Using this we can calculate the probability that a randomly selected vertex in a randomly selected Schroeder tree is a leaf.
\begin {corollary}
The probability that a randomly selected vertex in a random Schroeder tree is a leaf approaches \begin{equation}\dfrac{1+\sqrt{2}}{\sqrt{2}(3+\sqrt{8})}\approx .293\end{equation} as the number of leaves goes to infinity.
\end{corollary}
\begin{proof}
We compare the number of leaves which was calculated in Section 5 Subsection 1 to the number of vertices that was calculated in the previous theorem.
\end{proof}
This is also the probability that a vertex has subtree size 1 since every leaf has a subtree of size one and the only vertices with a size one subtree are those vertices which are leaves.
\subsection{$k$ verties in the subtree}
We will first handle the case where the subtree has $k$ vertices in it.
\begin{lemma}
The generating function for the number of vertices which have size $k$ subtrees is \begin{equation}\dfrac{(R_k(x))(3-x+\sqrt{1-6x+x^2})}{4\sqrt{1-6x+x^2}},\end{equation} where $R_k(x)$ is the generating function for the number of trees with $k$ vertices.
\end{lemma}
\begin{proof}
This follows from the relationship \begin{equation}T_k(x)=R_k(x)+2T_k(x)S(x)+3T_k(x)(S(x))^2+4T_k(x)(S(x))^3+... \hspace{.1 in}.\end{equation}  The $R_k(x)$ accounts for all cases where the root has a subtree of size $k$.  After accounting for that we can cut off the root and that will break down our tree into a forest of some number of trees $m$.  We look at a specific subtree in that forest.  The number of vertices with a size $k$ subtree in that tree is given by $R_k(x).$  We now multiply that by the total number of configurations which is $S(x)$ for each of the remaining trees.  This means $T_k(x)S(x)^{m-1}$ gives the number of vertices that one of the $m$ subtrees contributes and then we multiply that by $m$ to get the total number of all of the vertices in the forest.  Looking at that we get 
\newline
\begin{equation}T_k(x)=R_k(x)+T_k(x)\dfrac{S(x)}{(1-S(x))^2}.\end{equation}  Substituting in what we know we get the desired generating function.
\end{proof}
As an aside we could have also used this technique to find the generating function for the number of vertices and the number of leaves.  In that case for leaves we would have \begin{equation}L(x)=x\dfrac{3-x+\sqrt{1-6x+x^2}}{4\sqrt{1-6x+x^2}},\end{equation} since the only way the root can be a leave is if it is the only vertex and that tree has generating function $x$.  Similarly the number of vertices is given by \begin{equation}V(x)=S(X)\dfrac{3-x+\sqrt{1-6x+x^2}}{4\sqrt{1-6x+x^2}},\end{equation} since the root is always a vertex.
\newline
Now all we need to do is find $R_k(x)$.  We can do this using a bivariate generating functions.
\begin{lemma}
Let $R(x,y)$ be the bivariate generating function where $x$ is indexed by the number of leaves in the tree and $y$ is by the number of total vertices in the tree.  With that we have \begin{equation}R(x,y)=\dfrac{1+xy-\sqrt{(xy)^2+2xy+1-4xy(y+1)}}{2y+2}.\end{equation}
\end{lemma}
\begin{proof}
This is the bivariate generating function we got when trying to find the number of vertices and is calculated in exactly the same way.
\end{proof}
Now that we have that generating function $R_k(x)$ is simply the coefficent of $y^k$ in this expression which will be a generating function in $x$.
We can use the probabilities we already have, the generating functions, and Bender's lemma to calculate the probabilities and doing that we get.  It is relatively straight forward to check that the requirements of Bender's Lemma are satisifed.  $R_k(x)$ is our $A(x)$ and since it is a polynomial for all $k$ we know that it has an infinite radius of convergence and since $R_k(x)$ is also a polynomial with non-negative coefficents we also have that $R_k(b) \neq 0$ for all $k \neq 2$ in the case where $k=2$ the polynomial is identically $0$ so we cannot apply Benders Lemma, but it is unecessary since the number of subtrees of size $2$ is 0.  A tree cannot contain exactly two vertices.
\begin{center}
 \begin{tabular}{|c | c| c|} 
 \hline
k & $R_k(x)$ & Probability the subtree\\ &&has $k$ vertices $\approx$\\ [.5ex] 
 \hline
 1 & $x$ & .2929\\ 
 \hline
 2 & 0 & 0 \\
 \hline
3 & $x^2$ & 0.0503 \\
\hline
4 & $x^3$ & 0.0086\\
\hline
5 & $2x^3+x^4$ & 0.0187\\
\hline
6 & $5x^4+x^5$ & 0.0076\\
\hline
7 & $5x^4+9x^5+x^6$ & 0.0097\\
\hline
\end{tabular}
\end{center}
\subsection{$k$ leaves in the subtree}
We will see in many ways that htis is the much easier case to deal with and arguably the more natural way to define size $k$ subtrees  for Schroeder trees.  From the previous section we already  know that
\begin{lemma}
The generating function for the number of vertices which have $k$ leaves in their subtrees is \begin{equation}(L_k(x))\dfrac{3-x+\sqrt{1-6x+x^2}}{4\sqrt{1-6x+x^2}},\end{equation} where $L_k(x)$ is the generating function for the number of trees with $k$ leaves.
\end{lemma}
This is exactly as we had in the previous situation and the proof is exactly the same.  The only thing left to do is calculate $L_k(x)$ but that is just $s_k x^k$ since the number of trees where the root has $k$ leaves is just $s_k$.  From that we can generate the following table of probabilities.
\begin{center}
 \begin{tabular}{|c | c| c|} 
 \hline
k & $L_k(x)$ & Probability the subtree\\&& has $k $ leaves $\approx$\\ [.5ex] 
\hline
 1 & $x$ & 0.2929\\ 
 2 & $x^2$ & 0.0503 \\
3 & $3x^3$ & 0.0259 \\
4 & $11x^4$ & 0.0163\\
5 & $45x^5$ & 0.0114\\
6 & $197x^6$ & 0.0086\\
7 & $903x^7$ & 0.0067\\
\hline
\end{tabular}
\end{center}
\section{Further Directions}
The technique used in this paper to calculate the probability a vertex has a subtree of size $k$ very easily generalizes to other types of trees.  In the case of subtrees having $k$ vertices we merely need to find the generating function for a vertex having some property and the generating function for the number of trees where the root has a size $k$ subtree.  This makes it incredibly easy to deal with trees counted by the number of vertices because the generating function for the number of roots with $k$ vertices in their subtree is merely $T_k x^k$.  Similarly for trees counted by the number of leaves it would be equally simple to find the proportion of vertices with $k$ leaves in their subtree.
\section{Open questions}
If we look at Motzkin Trees again we see that the probability a vertex has $k$ vertices in its subtrree is merely $M_k(\frac{1}{3})^k$ and if we sum this over all $k$ we get 1.  So we say that $k$ vertices in the subtree is a tight statistic for Motzkin trees.  We also know that there exist classes of trees for which the statistic is not tight the most obvious being the class of trees for which each parent is allowed only 1 child.  In such a tree the asymptotic probability that an arbitrary vertex has a size $k$ subtree is 0 for all $k$ as the number of vertices goes to $\infty$.  Two interesting questions we could ask are is this stastic tight for Schroeder trees and the other trees studied in this paper and in general is there any way to classify trees where this statistic is tight.  Also we could ask are there any non degenerate tree families for which this statistic is not tight.
\vspace {.1 in}
\newline
For both Motzkin Trees and ordered trees the probability a vertex has $k$ vertices in its subtree is a strictly decreasing sequence in $k$.  For full binary trees we see a sequence of 0 and a strictly decreasing sequence, this makes sense as a subtree can never have an even number of vertices.  The sequence for Schroeder trees is more interesting.  It is not strictly decreasing and it is not the merging of two strictly decreasing sequences, the 2 vertex example gets in the way.  However, it does seem to oscilate and other than that single anamoly with 2 vertices it seems to be two strictly decreasing sequences.  It would be interesting to see if this trend holds for larger values of $k$ or is there some point where it becomes strictly decreasing or perhaps there is no rhythm to the sequence.

\end{document}